%% file: paper_fixed_point.tex
\documentclass{amsart}
\usepackage{amssymb,amsmath,amsthm,amsfonts,epsfig,graphicx,amsrefs}
\usepackage[utf8]{inputenc}
\usepackage[T1]{fontenc}
\usepackage[english]{babel}
\usepackage{hyperref}

\hypersetup{
    colorlinks,
    citecolor=black,
    filecolor=black,
    linkcolor=black,
    urlcolor=black
}

\theoremstyle{plain}
\newtheorem{theorem}{Theorem}[section]

\newtheorem{lemma}[theorem]{Lemma}
\newtheorem{proposition}[theorem]{Proposition}
\theoremstyle{definition}
\newtheorem{remark}[theorem]{Remark}
\newtheorem{definition}[theorem]{Definition}

\newcommand{\R}{\mathbb R}

\newcommand{\C}{\mathcal C}
\newcommand{\Z}{\mathbb Z}
\newcommand{\N}{\mathbb N}
\renewcommand{\epsilon}{\varepsilon}

\title[A fixed point thm. for plane homeos with the top. shadowing property]{A fixed point theorem for plane homeomorphisms with the topological shadowing property.}
\date{\today}

\author[G.~Cousillas]{Gonzalo Cousillas}
\address{Instituto de Matemática y Estadística ``Rafael Laguardia''\\
Facultad de Ingeniería, Universidad de la Rep\'ublica\\
Montevideo, Uruguay.}
\email{gcousillas@fing.edu.uy}

\begin{document}

\begin{abstract}
In this paper we show that every homeomorphism of the plane with the topological shadowing property has a fixed point. Also, we show that a linear isomorphism of an Euclidean space has the topological shadowing property if and only if the origin is an attractor or a repeller.
\end{abstract}
\maketitle

\section{Introduction}

The shadowing property  was introduced in the works of Anosov and Bowen and plays an important role in several branches of dynamical systems' theory. Roughly speaking, a system satisfies the shadowing property if every approximate orbit (pseudo-orbit) can be traced by a true one. Numerical simulations of dynamical systems always produce approximate orbits. Thus, systems with the shadowing property are those in which numerical simulations does not introduce unexpected behaviour, numerical orbits are followed by theoretical ones. 
For compact metric spaces the shadowing property does not depend on the choice of the metric. 
But as will see in Remark \ref{met_hyp},  even for a hyperbolic linear isomorphism the 
shadowing property may depend on the metric. 
In \cite{DLRW13} it was introduced a notion of shadowing property for topological spaces, the idea is to change the constants in the definition by neighborhoods of the diagonal in the product space.
Even if the space is non neccesarily compact, this definition has the remarkable property of being a conjugacy invariant. 

In this paper we study the topological shadowing property on euclidean spaces, mainly on the plane.
Our main result is the following:
\begin{theorem}\label{main:teo} Every orientation preserving planar homeomorphism with the topological shadowing property has a fixed point.\end{theorem}

The idea of the proof relies on Brouwer's plane translation theorem, and the fact that  a translation  does not satisfy the topological shadowing property (see  Lemma \ref{traslacion}).
The main problem we have to deal with is that even if a restriction of a homeomorphism to an invariant subset does not satisfy the shadowing property, the homeomorphism itself may satisfy it: it is possible that the orbit 
shadowing a pseudo orbit belongs to the complement of the invariant set. Thus we adopt another approach and build around  Brouwer's lines and translation domains.

The paper is organized as follows, in Section \ref{preliminaries} we set the  definition of topological shadowing and some preliminary results.  In Section \ref{ej:plano} we exhibit explicit examples,   we obtain that an homothety (by an homothety we mean the map $x\mapsto kx$, $|k|\neq 1$) satisfies the topological shadowing property but the  map $f(x,y)=(2x,\tfrac{1}{2}y)$  does not. In Section \ref{sec:mainteo} we prove Theorem \ref{main:teo}.

\section{Preliminaries}\label{preliminaries}

Throughout this paper we  consider $X$  a metric space  and  $f:X\to X$ a homeomorphism.  The orbit of $x\in X$ is the set $\mathcal{O}(x)=\{f^n(x)\}_{n\in\Z}$ and $x\in X$ is a fixed point for $f$ if  $f(x)=x$. 
Let $\alpha$ and $\beta$ be  positive real numbers. A sequence  $\{x_n\}_{n\in\Z}\subset X$  is  a $\beta$-pseudo-orbit if it verifies $d(f(x_n), x_{n+1})<\beta$ for every $n\in\Z$. A sequence $\{x_n\}_{n\in\Z}$  is  $\alpha$-shadowed by  an orbit if there exists  $y\in X$ such that  $d(f^n(y), x_n)<\alpha$ for every $n\in \Z$.
A homeomorphism $f:X\to X$  satisfies the {\it metric shadowing property} if given $\alpha>0$, there exists $\beta>0$ such that every $\beta$-pseudo-orbit is $\alpha$-shadowed by an orbit. 

Denote by $\mathcal{C}^+$ the set $\{\epsilon:X\to \R^+; \epsilon \text{ is continuous}\}$.
Let $\epsilon$ and  $\delta$ be maps in $\mathcal{C}^+$, a sequence $\{x_n\}_{n\in\Z}\subset X$ is  a $\delta$-pseudo-orbit if it verifies $d(f(x_n), x_{n+1})< \delta(f(x_{n}))$ for every $n\in \Z$. 
A sequence $\{x_n\}_{n\in\Z}$ is $\epsilon$-shadowed  by an orbit if there exists $y\in X$ such that  $d(f^n(y), x_n)< \epsilon(x_n)$ for every $n\in \Z$.
\begin{definition}\label{TS}
A homeomorphism $f:X\to X$ satisfies the {\it topological shadowing property} if for every map $\epsilon$ in $\mathcal{C}^+$, there exists a map $\delta$ in $\mathcal{C}^+$ such that every $\delta$-pseudo-orbit is $\epsilon$-shadowed  by an orbit.
Denote  $$TSP:=\{f:X\to X,  f \text{ satisfies the topological shadowing property}\}.$$
\end{definition}

\begin{proposition}\label{equivalence}
If $X$ is a locally compact metric space, definition \ref{TS} is equivalent to the definition given in \cite{DLRW13} which is: 
\begin{itemize}
 \item Given  neighborhoods $E$ and  $D$ of $\Delta_X$, a $D$-pseudo-orbit for $f$ is a sequence $\{x_n\}_{n\in\Z}\subset X$ such that $(f(x_n), x_{n+1})\in D$ for all $n\in \Z$.
 A $D$-pseudo-orbit is $E$ shadowed by an orbit if there exists $y\in X$ such that $(f^n(y), x_n)\in E$ for all $n\in \Z$.
 A homeomorphism $f$ satisfies the topological shadowing property if for every neighborhood $E$ of $\Delta_X$ there exists a neighborhood $D$ of $\Delta_X$ such that every $D$-pseudo-orbit is $E$ shadowed by an orbit.
\end{itemize}
\end{proposition}
The idea for proving this proposition appears in  \cite[Proposition 34]{DLRW13}.
\begin{proof} We claim that for every open neighborhood $E$ of $\Delta_X$ there exists a map $\epsilon\in \mathcal{C}^+$ such that $U_\epsilon=\{(x,y)\in X\times X: d(x,y)<\epsilon(x)\}\subset E$.
 
Let $E$ be a neighborhood of $\Delta_X$ such that $E[x]=\{y\in X: (x,y)\in E\}\subsetneq X$.
Define a function $h:X\to \R^+$ by $h(x)=d(x, X-E[x])$. 
Now define $\epsilon:X\to\R^+$ such that $\epsilon(x)=\inf\{h(y)+d(x,y):y\in X\}$.
By \cite[Theorem 67.2]{Zi}  $\epsilon$ is continuous.
Let us see that  $U_\epsilon[x]\subset E[x]$ for every $x\in X$.
Let $y\in U_\epsilon[x]$, then $d(x,y)<\epsilon(x)=\inf\{h(z)+d(x,z): z\in X\}$.
Choose $z=x$, then $d(x,y)<\epsilon(x)\leq h(x)=d(x,X-E[x])$.
This implies that $y\in E[x]$ and the claim follows.
Conversely, a map $\epsilon\in \mathcal{C}^+$ defines a neighborhood of $\Delta_X$ if we take the set $U_\epsilon=\{(x,y)\in X\times X: d(x,y)<\epsilon(x)\}$. 
 
 Taking this into account the equivalence is straightforward. \end{proof}

\begin{remark} In \cite{Lee} for shadowing a pseudo-orbit the authors consider $d(f^n(y), x_n)<\epsilon(f^n(y))$ i.e., the distance between $f^n(y)$ and $x_n$ depends on the value of $\epsilon$ in $f^n(y)$. We show it is equivalent if we consider  $\epsilon(x_n)$ instead of $\epsilon(f^n(y))$.  
\end{remark}
\begin{proof}
 Let $\epsilon$ be an arbitrary element of $\mathcal{C}^+$, take $U_\epsilon$ defined above.
 Note that for this $U_\epsilon$ there exists a symmetric neighborhood $\tilde{U}_\epsilon$ of $\Delta_X$ such that $\tilde{U}_\epsilon\subset U_\epsilon$. 
 By the previous proposition there exists a neighborhood $D$ of $\Delta_X$ such that every $D$-pseudo-orbit is $\tilde{U}_\epsilon$-shadowed.
 Applying again the previous proposition, there exists  $\delta\in\mathcal{C}^+$ such that $U_\delta\subset D$.
 Then every $\delta$-pseudo-orbit is $\tilde{U}_\epsilon$-shadowed, i.e.,  $(f^n(y), x_n)\in \tilde{U}_\epsilon$. 
 Since $\tilde{U}_\epsilon$ is symmetric, we have $( x_n, f^n(y))\in \tilde{U}_\epsilon$. 
 Thus $d(x_n, f^n(y))<\epsilon(f^n(y))$.
\end{proof}

 

Next proposition is a useful tool to create  examples of homeomorphisms satisfying the topological shadowing property.
\begin{proposition}\cite[Proposition 13]{DLRW13} \label{conjugaciones}
Let $X$ be a first countable, locally compact, paracompact and Hausdorff  space.
\begin{enumerate}
\item $f\in TSP$ if and only if $f^k\in TSP$ for  all nonzero $k\in \Z$.
 \item If $f\in TSP$ and $f$ is conjugate to $g$, then $g\in TSP$.
 \end{enumerate}
\end{proposition}

\section{Some examples on the plane}\label{ej:plano}
 We show in this section that an homothety belongs to $TSP$, however a linear transformation with two different eigenvalues, one  of module greater than $1$  and another lower that $1$ does not belong to $TSP$.

\begin{theorem}\label{hyperb}
 Let $g:\R^2\to \R^2$ be in the conjugacy class of  the map $f:\R^2\to\R^2$  such that $f(x,y)=(2x,\tfrac{1}{2}y)$. Then  $g\notin TSP$.
\end{theorem}

\begin{proof}
Since  topological shadowing property is conjugacy invariant, we will take the map $f(x,y)=(2x,\tfrac{1}{2}y)$.
Let $(x,y)\in\R^2$, consider $\|(x,y)\|=\max\{|x|,|y|\}$.
Let $\epsilon\in\mathcal{C}^+$ be such that $\epsilon(x,y)=2^{-\|(x,y)\|}$.
Take $(x_0,0)\in\R^2$ with $x_0>0$.
Given an arbitrary $\delta\in \mathcal{C}^+$ take $(x_0,y_0)$, $y_0>0$, such that $d((x_0,y_0),(x_0,0))<\delta(x_0,0)$.
Consider the following $\delta$-pseudo-orbit $\{(x_n, y_n)\}_{n\in\Z}$:
$$(x_n,y_n)=\begin{cases} f^n(x_0,0)& \mbox{ if } n\geq 0\\ f^n(x_0,y_0) &\mbox{ if }n< 0\end{cases}.$$
Suppose that the shadowing property holds.
Let  $(\bar{x},\bar{y})\in\R^2$ be a point such that $$d(f^n(\bar{x},\bar{y}), (x_n,y_n))<\epsilon(x_n,y_n),\;n\in\Z.$$
Since  $(x_n,y_n)=2^n(x_0,0)$ for $n>0$, we have \begin{eqnarray*} d(f^n(\bar{x},\bar{y}), (x_n,y_n))&=&d(f^n(\bar{x},\bar{y}), 2^n(x_0,0))\\&=&\max\{ 2^n|x_0-\bar{x}|, \tfrac{1}{2^n}|\bar{y}|\}.\end{eqnarray*}
If $\bar{x}\neq x_0$, then there exists $n_0$ such that  $d((x_n,y_n), f^n(\bar{x},\bar{y}))=2^n|\bar{x}-x_0|$ for $n>n_0$.
Since $\epsilon(x,y)=2^{-\|(x,y)\|}$, we have that $\epsilon(2^n(x_0,0))=\tfrac{1}{2^{2^n\|(x_0,0)\|}}$ and it follows that $2^n|\bar{x}-x_0|<\tfrac{1}{2^{2^n\|(x_0,0)\|}}$ for every $n>n_0$. 
Thus $\bar{x}=x_0$ and $d((x_n,0),f^n(\bar{x},\bar{y}))=\tfrac{1}{2^n}|\bar{y}|$.
So,   $$\tfrac{1}{2^n}|\bar{y}|=d(f^n(\bar{x},\bar{y}),(x_n,y_n)) <\epsilon(2^n(x_0,0))=\tfrac{1}{2^{2^n\|(x_0,0)\|}},\;\;\text{for } n>0.$$
Then $|\bar{y}|<\tfrac{2^n}{2^{{2^n}\|(x_0,0)\|}}\;$  for  $n>0$ and this implies $\bar{y}=0$.
We conclude that $(\bar{x},\bar{y})=(x_0,0)$ and the only orbit that is capable of shadowing this pseudo-orbit is the orbit of $(x_0,0)$.
Since the orbit of $(x_0,0)$  tends to $(0,0)$ when $n$ tends to $-\infty$ and  the pseudo-orbit tends to  $\infty$ when $n$ tends to $-\infty$,  we conclude that the orbit of $(x_0,0)$ cannot shadow $\{(x_n,y_n)\}_{n\in\Z}$.
Thus $f\notin TSP$.
\end{proof}

\begin{remark} In higher dimensions a similar result can be obtained.
Let $f:\R^d\to \R^d$ be a linear isomorphism such that there exists $\lambda, \mu\in Spec(f)$, $\lambda>1$, $0<\mu<1$, then $f\notin TSP$.
We just need to create a pseudo-orbit that comes to $0$ from the eigenspace asociated to $\mu$ and jumps to the eigenspace asociated to $\lambda$, then the $\epsilon$ map in $\C^+$ can be defined the same way as was done in the previous proposition.
\end{remark}

\begin{remark}\label{met_hyp}
With the usual metric a linear isomorphism satisfies the  metric shadowing property (see \cite{Om94}). However it is possible to choose another metric inducing the same topology such that  the metric shadowing property does not hold. Take for instance $h:\R_{\geq0}\to \R$ given by $h(r)=r+r^2$, then consider the points in polar coordinates and define $d'(p,q)=\|H(p)-H(q)\|$ where $H(r,\theta)=(h(r),\theta)$. This way if a linear isomorphism is in the conditions of the previous theorem, following the same ideas of the previous proof it can be shown that for $(\R^2, d')$ this map does not admit the metric shadowing property.
\end{remark}


\begin{definition} Let  $f:\R^d\to \R^d$ be a homeomorphism.
We say that $f$ satisfies the {\it forward topological shadowing property} if given $\epsilon\in \mathcal{C}^+$, there exists $\delta\in\mathcal{C}^+$ such that for every $\delta$-pseudo-orbit $\{x_n\}_{n\in\Z}$ there exists $y\in\R^d$ such that $d(f^{n}(y), x_n)\leq \epsilon(x_n)$, $n\geq 0$.
\end{definition}

The following lemma plays an important role for proving Theorem \ref{homothety}.  
\begin{lemma}\label{lemma_shadow}
Let  $f:\R^d\to\R^d$ be a homeomorphism with the forward topological shadowing property. Then $f\in TSP$. 
\end{lemma}

\begin{proof} 
Given $\epsilon\in \mathcal{C}^+$, take $\delta\in \mathcal{C}^+$ of the definition of the forward topological shadowing property and $\{x_n\}_{n\in\Z}$ a $\delta$-pseudo-orbit.
There exists $y_0$ such that $d(x_n, f^n(y_0))\leq\epsilon(x_n)$ for  $n\geq 0$, in particular $y_0\in \overline{B(x_0, \epsilon(x_0))}$.
Let $\{z_n^{(-1)}\}_{n\in\Z}$ be the $\delta$-pseudo-orbit defined as   $z^{(-1)}_n=x_{n-1}$.
Then there exists $y_{-1}$ such that $d(z^{(-1)}_n, f^n(y_{-1}))\leq\epsilon(z^{(-1)}_n)$ for  $n\geq0$.
In particular $f(y_{-1})\in \overline{B(x_0,\epsilon(x_0))}$.
Let $k\in\N$, and let $\{z_n^{-k}\}_{n\in\Z}$ defined by $z^{(-k)}_n=x_{n-k}$.
Then  there exists $y_{-k}$ such that $$d(z^{(-k)}_n, f^n(y_{-k}))<\epsilon(z^{(-k)}_n)\;\;\text{for } n\geq 0,\;\text{and } f^k(y_{-k})\in \overline{B(x_0, \epsilon(x_0))}.$$
 Thus we have the sequence $\{f^k(y_{-k})\}_{k\in\N}\subset \overline{B(x_0,\epsilon(x_0))}$.
 Let $\tilde{y}$ be a limit point of $\{f^k(y_{-k})\}_{n\in\N}$.
 Let us show that the orbit of $\tilde{y}\;$ $\epsilon$-shadows $\{x_n\}_{n\in\Z}$.
 Let $l\in \Z$, and $k_i\in\N$ be such that $l+k_i>0$.
 Then 
\begin{eqnarray*}
 d(x_l,f^l(\tilde{y}))&=&d(x_l, f^l(\lim_{k_i\to\infty}f^{k_i}(y_{-k_i})))\\&=&\lim_{k_i\to\infty}d(x_l,f^{l+k_i}(y_{-k_i}))\\
&\leq&\epsilon(x_l)
\end{eqnarray*}
This completes the proof.
\end{proof}

Although our main interest is to work in the plane, the next result is valid in $\mathbb{R}^d$.

\begin{theorem}\label{homothety}
Let $g:\R^d\to\R^d$ be in the conjugacy class of an homothety.
Then $g\in TSP$.
\end{theorem}
\begin{proof}
Since  the topological shadowing property is conjugacy invariant, we take $f:\R^d\to\R^d$ such that $f(x)=2x$.
The main idea is to find $\delta\in \mathcal{C}^+$ in order to obtain just two kinds of  $\delta$-pseudo-orbits, those that remain close to the origin and those that tend (in norm) to $\infty$ as $n$ tends to $+\infty$.
We will forward shadow those pseudo-orbits and then apply Lemma \ref{lemma_shadow}.

Formally, let $\epsilon\in\mathcal{C}^+$ and $r_0=\epsilon(0)$ and  $m=\min\{\epsilon(x): x\in \overline{B(0,r_0)}\}$.
Take $\delta\in\mathcal{C}^+$  bounded such that \begin{equation}\label{cond:1} \delta(x)<\epsilon(x)\;\text{ for all } x\in\R^d, \text{ and }\end{equation}  
\begin{equation}\label{cond:2} \delta(x)<\begin{cases}\;\;\;m &\mbox{ if } x\in\overline{B(0,r_0)}\\
\;\;\tfrac{\|x\|}{4}&\mbox{ if }  x\notin\overline{B(0,r_0)} 
           \end{cases}\end{equation}

In this paragraph we are going to  show that for this $\delta$, there exist only two kinds of $\delta$-pseudo-orbits $\{x_n\}_{n\in\Z}$, the ones such that  $x_n\in \overline{B(0,r_0)}$ for all $n\in\Z$, and the ones such that $\|x_n\|\to +\infty$ as $n\to +\infty$. 
Let us show the existence of the first  kind of $\delta$-pseudo-orbits.
Since $0$ is fixed by $f$, $\delta$ is continuous and strictly positive in $\R^d$, we can take $x_0\in\overline{B(0,r_0)}$ close enough to $0$ such that $f(x_0)$ verifies $d(0,f(x_0))<\delta(f(x_0))$.
Then  $x_0\in B(f(x_0),\delta(f(x_0)))$.
We can choose $x_1$ to be $x_0$.
Repeating the previous argument for every $n\geq 1$ we have that $x_n\in B(0,r_0)$ for all $n\geq 0$.
For $n<0$, take $x_n=f^n(x_0)$.
Thus $x_n\in \overline{B(0, r_0)}$ for all $n\in \Z$. 
Now we show the existence of the other kind of $\delta$-pseudo-orbits.
Let $\{x_n\}_{n\in\Z}$ be a $\delta$-pseudo-orbit such that there exists $x_i \notin \overline{B(0,r_0)}$.
Suppose for simplicity $i=0$.
Since    $\delta(x)<\tfrac{\|x\|}{4}$ if $x\notin\overline{B(0,r_0)}$, we have that $$x_1\in B(f(x_0),\delta(f(x_0)))\subset B\left(f(x_0),\tfrac{\|f(x_0)\|}{4}\right).$$
As $B(f(x_0),\|f(x_0)\|/4)=B\left(2(x_0), \|x_0\|/2\right)$, then $\tfrac{3}{2}\|x_0\|<\|x_1\|$,  and $x_1\notin \overline{B(0,r_0)}$. 
Similarly for $x_n$, $n>0$, we prove that $x_n\notin \overline{B(0,r_0)}$ and we obtain that $\tfrac{3}{2}\|x_n\|<\|x_{n+1}\|$.
Therefore  $\left(\tfrac{3}{2}\right)^n\|x_0\|<\|x_n\|$ for all $n>0$, and  $\|x_n\|\to+\infty$ as $n\to +\infty$.

It is immediate that the fixed point $0$ $\epsilon$-shadows the first kind of pseudo-orbits. 
Let $\{x_n\}_{n\in\Z}$ be a $\delta$-pseudo-orbit that tends to infinity.
Then $x_1\in B(f(x_0),\delta(f(x_0)))$, so  $x_1=2x_0+r_1$, where  $\|r_1\|<\delta(f(x_0))$.
 For $n=2$ note that $f(x_1)=2^2x_0+2r_1$ and $x_2\in B(f(x_1),\delta(f(x_1)))$, so $x_2=2^2x_0+2r_1+r_2$, where $\|r_2\|<\delta(f(x_1))$.
 Similarly for an arbitrary $n>0$,  $$x_n=2^nx_0+2^{n-1}r_1+\ldots +2r_{n-1}+r_n$$ where  $\|r_i\|<\delta(f(x_{i-1}))$.
Now  consider the sequence $\left\{\tfrac{x_n}{2^n}\right\}_{n\in \N}$, then \begin{equation} 
                                                                        \tfrac{x_n}{2^n}=x_0+2^{-1}r_1+\ldots+2^{1-n}r_{n-1}+2^{-n}r_n \label{eq:po}
                                                                       \end{equation}
Since $\delta$ is bounded, we have that there exists $k>0$ such that $\|r_i\|<\delta(f(r_{i-1}))<k$.
This implies that the sequence in \eqref{eq:po} converges.
Then define $\bar{x}$ as  $$\bar{x}=x_0+\sum_{i=1}^\infty \tfrac{r_i}{2^i}.$$
 
Next we will show that the orbit of $\bar{x}$ forward $\epsilon$-shadows $\{x_n\}_{n\in\Z}$.  
Let $l>0$,
\begin{equation}\label{eq2}
  \begin{split} 
  d(f^l(\bar{x}),x_l)&=  \left\|2^lx_0+\sum_{i=1}^\infty \tfrac{r_i}{2^{i-l}}-2^lx_0-\sum_{i=1}^l\tfrac{r_i}{2^{i-l}}\right\| \\
	      &= \left\|\sum_{i=l+1}^\infty \tfrac{r_i}{2^{i-l}}\right\|  \\
	      &\leq \sum_{i=l+1}^\infty \tfrac{\|r_i\|}{2^{i-l}}  \\
	      &<\sum_{i=l+1}^\infty \tfrac{\delta(f(x_{i-1}))}{2^{i-l}}
  \end{split}		
\end{equation}
If $x_0\in\overline{B(0,r_0)}$, then there exists $i_0$ such that $x_i\in\overline{B(0,r_0)}$ for $0\leq i<i_0$ and $x_{i_0}\notin \overline{B(0,r_0)}$. 
For $0\leq l<i_0$, $x_l\in \overline{B(0,r_0)}$ and then $\epsilon(x_l)\geq m$.
Let us assume that  \begin{equation}\label{eq5}  \delta(x)<m \text{ for all }x\in\R^d, \end{equation}
Then by \eqref{eq2}
$$d(f^l(\bar{x}),x_l)<\sum_{j=l+1}^\infty \tfrac{\delta(f(x_{j-1}))}{2^{j-l}}\stackrel{(\ref{eq5})}{\leq} \sum_{j=l+1}^\infty \tfrac{m}{2^{j-l}}\leq m\leq \epsilon(x_l).$$ 
For $l\geq i_0$, we will suppose $\delta$ satisfies \begin{equation}\label{cond:4} \delta(x)>\delta(x') \text{ if } \|x\|<\|x'\|.\end{equation}
Thus  we have $\delta(f(x_{j-1}))<\delta(x_{j-1})$ and since $\|x_n\|_{n\in\Z}$ is an increasing sequence for $n>i_0$, we also have that $\delta(x_n)>\delta(x_m)$ for $n<m$.
Then again by \eqref{eq2}
\begin{eqnarray*}
d(f^l(\bar{x}),x_l)&<& \sum_{i=l+1}^\infty \tfrac{\delta(f(x_{i-1}))}{2^{i-l}}\\
                   &\leq&\sum_{i=l+1}^\infty \tfrac{\delta(x_{i-1})}{2^{i-l}}\\
                   &\leq& \delta(x_l)\left(\sum_{i=1}^\infty \tfrac{1}{2^i}\right)\\
	           &\leq& \epsilon(x_l)
\end{eqnarray*}
In the case that $x_0\notin \overline{B(0,r_0)}$,  $x_l\notin\overline{B(0,r_0)}$ for $l>0$ and with similar  arguments we can prove that  $d(f^l(\bar{x}), x_l)\leq \epsilon(x_l)$ for $l\geq 0$.
Then,  for that $\epsilon\in\mathcal{C}^+$ we required  $\delta$ to verify conditions (\ref{cond:1}), (\ref{cond:2}), (\ref{eq5}), (\ref{cond:4}).
Such function $\delta$ exists: condition (\ref{eq5}) says that $\delta$ is bounded.
Moreover, conditions (\ref{cond:1}) and (\ref{cond:2}) say that $\delta(x)<\min\left\{\epsilon(x),\tfrac{\|x\|}{4}\right\}$ and as these two functions are continuous and strictly positive,  $\delta$ may be strictly positive as well.
Finally condition (\ref{cond:4}) says that $\delta$ is decreasing in norm.  
Then for this  $\delta\in\mathcal{C}^+$,  every $\delta$-pseudo-orbit $\{x_n\}_{n\in\Z}$ is forward $\epsilon$-shadowed.

We finish the proof applying Lemma \ref{lemma_shadow}  to obtain that $f\in TSP$.       
\end{proof}

By a {\it reverse homothety} we mean the map $z\to \bar{z}/2$, $z\in \mathbb{C}$.  Applying the same ideas of the proof of the previous theorem, we can conclude that any homeomorphism conjugate to a reverse homothety satisfies the topological shadowing property.

In 1934 Kerekjarto's proved that a planar homeomorphism with an asymptotically stable fixed point is conjugate, on its basin of attraction, to one of the maps $z\to z/2$ or $z\to \overline{z}/2$, $z\in \mathbb{C}$, according to whether the homeomorphism preserves or reverses orientation (see \cite{GGM16}).
In case the basin of attraction is the whole plane we conclude, by the previous theorem and Theorem \ref{conjugaciones}, that a planar homeomorphism with an asymptotically stable fixed point satisfies the topological shadowing property.

\section{Main theorem: existence of a fixed point}\label{sec:mainteo}

In this section we  prove the existence of a fixed point for an orientation preserving planar homeomorphism with the topological shadowing property.

\begin{lemma}\label{traslacion}
 Let  $T:\R^d\to \R^d$ be such that $T(x)=x+e_1$, where $e_1=(1,0,\ldots, 0)$.  Then  $T\notin TSP$.
\end{lemma}

\begin{proof}
Let  $\epsilon\in \mathcal{C}^+$ be such that $\epsilon(x)\to 0$ as $\|x\|\to \infty$.
Let  $\delta$ be an arbitrary element of $\mathcal{C}^+$.
Consider $0=(0,\ldots,0)$ and $y=(0,y_2,\ldots, y_d)$ such that   $d(0,y)<\delta(0)$.
Let $\{x_n\}_{n\in\Z}$ be the pseudo-orbit defined as  $$x_n=\begin{cases} T^n(0)& \mbox{if } n\geq 0\\ T^n(y) &\mbox{if } n<0\end{cases}.$$
Suppose that the orbit of some $\bar{x}$ $\epsilon$-shadows this $\delta$-pseudo-orbit.
Then for $l>0$, $$\epsilon(x_l)>d(T^l(\bar{x}), x_l)=d(\bar{x}+le_1,le_1)=\|\bar{x}\|.$$
Since $\epsilon(x_l)\to 0$ as  $l\to \infty$,  we conclude that $\bar{x}=0$.
But for $l<0$, if $\bar{x}$ $\epsilon$-shadows $\{x_l\}_{n\in\Z}$ then $\bar{x}=y$.
 We conclude that $\{x_n\}_{n\in\Z}$  is not $\epsilon$-shadowed, therefore $T\notin TSP$.
\end{proof}

If we suppose that $f:\R^2\to\R^2$ is a fixed point free orientation preserving  homeomorphism, then  Brouwer's plane translation theorem (see \cite{Br12}) asserts that for every $x\in\R^2$ there exists a properly embedded topological line $L$ (called {\it Brouwer line}), such that $x\in L$  and $L$  separates $f(L)$ from $f^{-1}(L)$.
Let $U$ be the open region whose boundary is $L\cup f(L)$, $U$  is called {\it translation domain for $f$} determined by $L$ (see Figure \ref{brou}).
\begin{figure}[htb] 
\centering 
\def\svgwidth{210pt} 
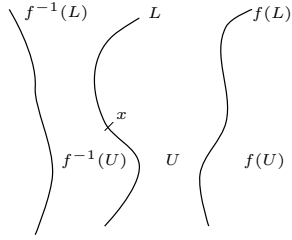
\caption{$L$ is a Brouwer line and $U$  the domain of translation determined by $L$.} 
\label{brou}
\end{figure}
Recall that an embedded topological line $L$ is the image of an injective continuous map $\varphi:\R\to \R^2$ such that $\varphi$ yields a homeomorphism between $\R$ and $\varphi(\R)=L$, where $L$ carries the subspace topology inherited from $\R^2$.
By  a properly embedding we mean that the function $\varphi$ is proper (preimages of compact sets   by $\varphi$ are compact).
If we consider $W=\bigcup_{n\in\Z} \overline{f^n(U)}$, then $W$ is a $f$-invariant open set and $f$ restricted to $W$ is conjugate to a translation on the plane \cite{Br12}.
Note that a  Brouwer line on $\R^2$ is  closed and that  the orbit of every point tends to infinity.
 
\begin{proof}[ Proof of  Theorem \ref{main:teo}:]
Let $f$ be a planar  homeomorphism with the topological shadowing property. 
Arguing by contradiction, suppose $f$ is  fixed point free. 
Let $\tilde{x}\in \R^2$, $L$  a Brouwer line such that  $\tilde{x}\in L$, let $U$ be the translation domain  determined by $L$ and $W=\bigcup_{n\in\Z} \overline{f^n(U)}$ the open set determined by $U$ such that $f|_W$ is conjugate to a translation.
In case $W=\R^2$, then $f$ is globally conjugate to a translation and by Lemma \ref{traslacion} $f\notin TSP$.

In case $W\neq\R^2$, then $\partial W\neq\emptyset$.
Note that the fact that $f|_W$ does not satisfy the shadowing property does not contradict the fact that $f$ does not satisfy the shadowing property.
Let $x_0\in\partial W$. Since $L$ is a Brouwer line, it is closed and then for every $x\in \mathcal{O}(x_0)$ there exists a neighborhood $U_{x}$ of $x$ such that \begin{equation}\label{cond:orbit} U_{x}\cap L=\emptyset.\end{equation}
On the other hand, since every point tends to infinity  if we consider the orbit of $x_0$, for every $x\in L$ we have that  \begin{equation}\label{cond:L} d(x,\mathcal{O}(x_0))>0.\end{equation}

We will find $\epsilon\in \mathcal{C}^+$ such that for an arbitrary $\delta\in \mathcal{C}^+$ there exists a $\delta$-pseudo-orbit that will not be $\epsilon$-shadowed. 
Let us first define $\epsilon$ in $\mathcal{O}(x_0)$. Consider $L'$ a Brouwer line for $x_0$ and $W'=\bigcup_{n\in\Z}\overline{f^n(U')}$, where $U'$ is the   translation domain determined by $L'$.
Then  $f|_{W'}:W'\to W'$ is conjugate to  a translation $T$ on $\R^2$. Let $h:W'\to \R^2$ be such conjugacy. 
Take the sequence $\{h(f^n(x_0))\}_{n\in\Z}=\{T^n(h(x_0))\}_{n\in\Z}\subset\R^2$.
Since $T$ is a translation on $\R^2$, the orbit of $h(x_0)$ tends to infinity, thus we take $\epsilon'\in \mathcal{C}^+$ such that $\epsilon'(T^n(h(x_0)))\to 0$ as $n\to+\infty$.
We  choose $\epsilon$ strictly positive on $\mathcal{O}(x_0)$ such that \begin{equation}\label{tr} B(f^n(x_0),\epsilon(f^n(x_0)))\subset \tilde{U}_{f^n(x_0)}\cap U_{f^n(x_0)}\end{equation} where $\tilde{U}_{f^n(x_0)}=h^{-1}\left(B(h(f^n(x_0)),\epsilon'(h(f^n(x_0))))\right)$ and $U_{f^n(x_0)}$  satisfying condition (\ref{cond:orbit}).
Now we define  $\epsilon$ on $L$.
Take $\epsilon:L\to\R$ such that $\epsilon(x)=\tfrac{d\left(x, \mathcal{O}(x_0)\right)}{2}$. 
Then $\epsilon$ is continuous and by condition (\ref{cond:L}) it is strictly positive.
We have defined $\epsilon$   continuous and strictly positive on  the closed set $\mathcal{O}(x)\cup L$.
Applying Tietze's extension theorem (see \cite{Ke}), we  can extend  $\epsilon$ continuously  to $\R^2$,  it is not difficult to see that it can be taken strictly positive as well.

Consider an arbitrary $\delta\in\mathcal{C}^+$. Since $x_0\in \partial W$,  we have that $B(x_0,\delta(x_0))\cap f^{m_0}(L)$ is nonempty for some $m_0$.
Suppose without lose of generality that $m_0>0$. Let $x_{m_0}\in B(x_0,\delta(x_0))\cap f^{m_0}(L)$ and consider the $\delta$-pseudo-orbit $\{x_n\}_{n\in\Z}$ defined as
$$x_n=\begin{cases} f^n(x_0) &\mbox{if } n\geq 0\\ f^n(x_{m_0})&\mbox{if } n<0\end{cases}.$$
For $n>0$, by condition (\ref{tr}) we have chosen $\epsilon$ on $\mathcal{O}(x_n)$ in order to obtain that the unique point whose orbit $\epsilon$-shadows $\{x_n\}_{n\in\Z}$ is that of $x_0$. 
On the other hand  $f^{-m_0}(x_{m_0})\in L$ and  $$B(f^{-m_0}(x_{m_0}),\epsilon(f^{-m_0}(x_{m_0})))\cap\mathcal{O}(x)=\emptyset.$$
We conclude that the orbit of $x_0$ can not shadow $\{x_n\}_{n\in\Z}$. Since $f\in TSP$, we have a contradiction and the proof of the theorem is complete.
\end{proof}



\section*{Acknowledgements}
The results of this paper are part of author's Magister's monograph under the guidance of Jorge Groisman. The author thanks the patience and hours of talks and discussions with him. 

\bibliographystyle{plain}

\bibliography{bibliografia}

\label{lastpage}
\end{document}

%% file: browline.pdf_tex
\begingroup%
  \makeatletter%
  \providecommand\color[2][]{%
    \errmessage{(Inkscape) Color is used for the text in Inkscape, but the package 'color.sty' is not loaded}%
    \renewcommand\color[2][]{}%
  }%
  \providecommand\transparent[1]{%
    \errmessage{(Inkscape) Transparency is used (non-zero) for the text in Inkscape, but the package 'transparent.sty' is not loaded}%
    \renewcommand\transparent[1]{}%
  }%
  \providecommand\rotatebox[2]{#2}%
  \ifx\svgwidth\undefined%
    \setlength{\unitlength}{362.83464355bp}%
    \ifx\svgscale\undefined%
      \relax%
    \else%
      \setlength{\unitlength}{\unitlength * \real{\svgscale}}%
    \fi%
  \else%
    \setlength{\unitlength}{\svgwidth}%
  \fi%
  \global\let\svgwidth\undefined%
  \global\let\svgscale\undefined%
  \makeatother%
  \begin{picture}(1,0.74999998)%
    \put(0,0){\includegraphics[width=\unitlength]{browline.pdf}}%
    \put(0.27624046,0.53244274){\color[rgb]{0,0,0}\makebox(0,0)[lb]{\tiny{\smash{$f^{-1}(L)$}}}}%
    \put(0.50000001,0.53125002){\color[rgb]{0,0,0}\makebox(0,0)[lb]{\tiny{\smash{$L$}}}}%
    \put(0.68749997,0.53125002){\color[rgb]{0,0,0}\makebox(0,0)[lb]{\tiny{\smash{$f(L)$}}}}%
    \put(0.53125002,0.26562495){\color[rgb]{0,0,0}\makebox(0,0)[lb]{\tiny{\smash{$U$}}}}%
    \put(0.34374999,0.26562495){\color[rgb]{0,0,0}\makebox(0,0)[lb]{\tiny{\smash{$f^{-1}(U)$}}}}%
    \put(0.67187498,0.26562495){\color[rgb]{0,0,0}\makebox(0,0)[lb]{\tiny{\smash{$f(U)$}}}}%
    \put(0.44283025,0.34992878){\color[rgb]{0,0,0}\makebox(0,0)[lb]{\tiny{\smash{$x$}}}}%
  \end{picture}%
\endgroup%